\documentclass[a4paper,oneside,15pt, reqno, psamsfonts]{amsart}
\usepackage{etoolbox}
\makeatletter
\patchcmd\maketitle
{\uppercasenonmath\shorttitle}
{}
{}{}
\patchcmd\maketitle
{\@nx\MakeUppercase{\the\toks@}}
{\the\toks@}
{}
{}{}
\patchcmd\@settitle{\uppercasenonmath\@title}{\Large}{}{}
\patchcmd\@setauthors
{\MakeUppercase{\authors}}
{\authors}
{}{}
\makeatother
\usepackage{amsmath,amssymb,amsthm, amsfonts}
\usepackage{mathrsfs}
\usepackage{color}
\usepackage{float}
\usepackage{url}
\usepackage{tikz-cd}
\usepackage{tikz}
\usepackage{amsthm}
\usetikzlibrary{arrows}
\usetikzlibrary{calc}
\usepackage{enumitem}
\usepackage{mathtools}
\usepackage[utf8]{inputenc}
\usepackage[T1]{fontenc}

\newtheorem{theorem}{Theorem}[section]

\newtheorem{question}[theorem]{Question}

\newtheorem{proposition}{Proposition}[section]
\newtheorem{corollary}[proposition]{Corollary}

\newtheorem{lemma}{Lemma}[section]
\newtheorem{remark}[lemma]{Remark}
\newtheorem{example}[lemma]{Example}
\numberwithin{equation}{section}
\newcommand{\beq}{\begin{eqnarray}}
	\newcommand{\eeq}{\end{eqnarray}}
\newcommand{\beqn}{\begin{eqnarray*}}
	\newcommand{\eeqn}{\end{eqnarray*}}
\newcommand{\rar}{\rightarrow}
\newcommand*{\Ge}{\geqslant}
\newcommand*{\Le}{\leqslant}

\usepackage[colorlinks=true]{hyperref}
\hypersetup{urlcolor=blue, citecolor=red , linkcolor= blue}
\usepackage[capitalise,noabbrev,nameinlink]{cleveref}      

\begin{document}
	\title[On the abscissas of a Dirichlet series and its subseries supported on prime factorization]
		{On the abscissas of a Dirichlet series and its subseries supported on prime factorization}
	
	\author[C. K. Sahu]{Chaman Kumar Sahu}
	\address{Department of Mathematics, Indian Institute of Technology Bombay,
		Mumbai 400076, India}
	\email{sahuchaman9@gmail.com, chamanks@math.iitb.ac.in}

%
	
	
	\maketitle
	%

	\begin{abstract}
		For a sequence $\{a_n\}_{n \Ge 1} \subseteq (0, \infty)$ and a Dirichlet series $f(s) = \sum_{n=1}^\infty a_n n^{-s},$ let $\sigma_a(f)$ denote the abscissa of absolute convergence of $f$ and let
		\beqn
		\delta_a(f): = \inf\Bigg\{\Re(s) : \sum\limits_{\substack{j= 1 \\\tiny{\mbox{\bf gpf}}(j) \Le p_n }}^\infty a_j j^{-\Re(s)} < \infty ~\text{for all}~ n \Ge 1\Bigg\},  
		\eeqn
		where $\{p_j\}_{j \Ge 1}$ is an increasing enumeration of prime numbers and $\text{\bf gpf}(n)$ denotes the greatest prime factor of an integer $n \Ge 2.$ 
		One significant aspect of these abscissas is their crucial role in analyzing the multiplier algebra of Hilbert spaces associated with diagonal Dirichlet series kernels.
        The main result of this paper establishes that  $\sigma_a(f)- \delta_a(f)$ can be made arbitrarily large, meaning that it can be equal to any non-negative real number. As an application, we determine the multiplier algebra in some cases and, in others, gain insights into the structure of the multiplier algebra of certain Hilbert spaces of Dirichlet series.
	\end{abstract}


	\section{Introduction}
	
	Let $\mathbb N,$~$\mathbb R$ and $\mathbb C$ denote the sets of positive integers, real and complex numbers, respectively. 
	For $s \in \mathbb C,$ let $\Re(s),$ $|s|,$ $\overline{s}$ and $\arg(s)$ denote the real part, the modulus, the complex conjugate and the argument of $s,$ respectively. 
	For $\rho \in \mathbb R,$ let $\mathbb H_\rho$ denote the right half-plane $\{s \in \mathbb C : \Re(s) > \rho\}.$ 
	
	A {\it Dirichlet series} is a series of the form
	\beqn
	f(s) = \sum_{j=1}^{\infty} a_{j} j^{-s},
	\eeqn
	where $a_j \in \mathbb C,~j \Ge 1$ and $s$ is a complex variable. 
	If $a_j = 1$ for all $j \Ge 1,$ then we have the Riemann zeta function, denoted by $\zeta.$   
	If $f$ is convergent at $s=s_{0}$, then it converges 
	uniformly throughout the angular region $\{s \in \mathbb C : |\arg(s-s_0)| < \frac{\pi}{2} -\delta\}$ for every positive real number $\delta < \frac{\pi}{2}.$ Consequently, $f$ defines a holomorphic function on $\mathbb H_{\Re(s_0)}$ (refer to \cite[Chapter~IX]{Ti} for the basic theory of Dirichlet series). 	
	
	The abscissa of {\it absolute convergence} $\sigma_{a}(f)$ of $f$ is the extended real number defined by
	\beqn
	\sigma_{a}(f) = \inf \Big\{\Re(s) : \sum_{n=1}^{\infty} |a_{n}| n^{-\Re(s)} < \infty \Big\}.
	\eeqn
	For an integer $n \Ge 2,$ let $\text{\bf gpf}(n)$ denote the greatest prime factor of $n,$ and let $\text{\bf gpf}(1) = 1.$ Define
	\beqn
	\delta_a(f) := \inf\Bigg\{\Re(s) : \sum\limits_{\substack{j= 1 \\\tiny{\mbox{\bf gpf}}(j) \Le p_n }}^\infty |a_j| j^{-\Re(s)} < \infty ~\text{for all}~ n \in \mathbb N\Bigg\},  
	\eeqn
	where $\{p_j\}_{j=1}^\infty$ is the increasing enumeration of prime numbers. For every integer $n \Ge 1,$ let
	\beqn
	f_n(s): = \sum\limits_{\substack{j= 1 \\\tiny{\mbox{\bf gpf}}(j) \Le p_n }}^\infty a_j j^{- s}.
	\eeqn
	Then, by definition, $\delta_a(f) = \sup_{n \Ge 1}\sigma_a(f_n)$. Since $\sigma_a(f_n) \Le \sigma_a(f_{n+1})$ for all integers $n \Ge 1$, it follows that
	\beq
	\label{eventually-equal-lemma}
	\delta_a(f) = \sup_{n \Ge m}\sigma_a(f_n), \quad \text{for any}~m \Ge 1.
	\eeq 
		Let $\mathcal D$ denote the set of all functions which are representable by a convergent Dirichlet series in some right half plane. 


If $f \in \mathcal D$ has all positive coefficients, then 
$\delta_a(f) \Le \sigma_a(f).$ One of the importance of these abscissas lies in the fact that they appear in the study of Hilbert space associated to diagonal Dirichlet series kernels and its multiplier algebra (see, \cite{Sa}, \cite[Theorem~3.1]{St} and Proposition~\ref{stetler_result_repeat}).
Thus, the following natural question arises.
	\begin{question}\label{Q-1}
		For $r \Ge 0,$ does there exist $f \in \mathcal D$ with positive coefficients such that $\sigma_a(f) - \delta_a(f) = r?$ 
	\end{question}
 If $r =1,$ Question~\ref{Q-1} has an affirmative answer. Indeed, consider $\zeta(s) = \sum_{n=1}^\infty n^{-s},$ which has all coefficients equal to $1.$ Then, $\sigma_a(\zeta) = 1,$ and for every $n \in \mathbb N,$ since the subseries 
	\beq \label{zeta-n}
\zeta_n(s): = 	\sum\limits_{\substack{j= 1 \\\tiny{\mbox{\bf gpf}}(j) \Le p_n }}^\infty j^{- s} = \prod_{j=1}^n \frac{1}{1-p_j^{-s}}
	\eeq
	converges on $\mathbb H_0$ (see \cite[Lemma~3.2]{St}) and diverges at $0$, it follows that $\delta_a(\zeta) = 0.$ Therefore, $\sigma_a(\zeta) - \delta_a(\zeta) = 1.$
  From the definitions of both abscissas, one might naturally expect that the difference between these abscissas can not be too large. However, in relation to Question~\ref{Q-1}, we establish a surprising result in this context, as presented below.
 
	\begin{theorem}\label{arbitrary-large} The following statements are true.
	\begin{itemize}
		\item [$(i)$] For every real number $r \Ge 0$, there exists a Dirichlet series $f$ with positive coefficients such that $\sigma_a(f)-\delta_a(f) = r$. 
		\item [$(ii)$] If $r \Ge 1,$ there exists  a Dirichlet series $f$ with multiplicative coefficients $($i.e.,  $a_{mn} = a_ma_n$ whenever $\gcd(m,n)=1)$ such that $\sigma_a(f)-\delta_a(f) = r$.
	\end{itemize}
	\end{theorem}
A proof of Theorem~\ref{arbitrary-large} is given in Section~\ref{S2}. Regarding Theorem~\ref{arbitrary-large}(i), we present two approaches for constructing Dirichlet series in the case $r=0$. The second approach appears quite natural, as it encompasses many interesting examples (see, Example~\ref{example-inte}).
	\section{Proof of Theorem~\ref{arbitrary-large}}\label{S2}

%

	\begin{proof}[\textbf{Proof of Theorem~\ref{arbitrary-large}}]
	(i) Let $r \Ge 0$ and consider the Dirichlet series $g$ defined as  
	\beqn
	g(s) = \sum_{n=1}^\infty \frac{p_n^{r-1-s}}{1 - p_n^{-s}},
	\eeqn
	where $\{ p_n \}_{n=1}^{\infty}$ denotes the increasing enumeration of all prime numbers. Define $g_n$ as the subseries of $g$ consisting of terms supported on integers that can be factorized using only the first $n$ prime numbers. Then, $g_n$ takes the form  
	\beqn
	g_n(s) = \sum_{j=1}^{n} \frac{p_j^{r-1-s}}{1 - p_j^{-s}}.
	\eeqn	
	Since $\sigma_a(g_n) = 0$ for every integer $n \Ge 1$, by \eqref{eventually-equal-lemma}, it follows that $\delta_a(g) = 0$.  
	Next, we determine $\sigma_a(g)$. To this end, observe that applying \cite[Eq.~(0.4)]{DGMS} yields  
	\beqn
	g(\sigma) \Le \frac{1}{1 - 2^{-\sigma}} \sum_{n=1}^{\infty} p_n^{r-1-\sigma} < \infty, \quad \text{for } \sigma > r.
	\eeqn
	This establishes that $\sigma_a(g) \Le r$.  
	Furthermore,  
	\beqn
	g(r) = \sum_{n=1}^{\infty} \frac{1}{p_n - p_n^{1-r}} \Ge \sum_{n=1}^{\infty} \frac{1}{p_n}.
	\eeqn
	By another application of \cite[Eq.~(0.4)]{DGMS}, it follows that $\sigma_a(g) = r$.
	Now, consider the Dirichlet series $f$ given by
	\beqn
	f(s) &=& \zeta(s-r+1) (1 +g(s)).
	\eeqn
	Since both $\zeta(\cdot-r+1)$ and $g$ converge absolutely on $\mathbb H_r$, it follows from \cite[Theorem~4.3.1]{QQ} that $\sigma_a(f) \Le r.$ Let $(c_m)_{m=1}^\infty$ be the positive sequence determined uniquely by the relation
	\beq
	\label{coeff-i}
	f(s) &=& \Big(\sum_{m=1}^\infty m^{r-1}m^{-s}\Big) \Big(1+ \frac{2^{r-1-s}}{1 - 2^{-s}} + \frac{3^{r-1-s}}{1 - 3^{-s}} + \frac{5^{r-1-s}}{1 - 5^{-s}} + \cdots \Big) \notag\\
	&=& \sum_{m=1}^\infty c_m m^{-s}, \quad s \in \mathbb H_r.
	\eeq 
	Then, $c_p = 2 p^{r-1}$ for any prime $p,$ so that
	\beqn
	f(r) \Ge \sum_{j=1}^\infty c_{p_j} p_j^{-r} = 2 \sum_{j=1}^\infty \frac{1}{p_j} = \infty,
	\eeqn
	which implies that $f$ diverges at $r$ and hence, $\sigma_a(f) = r.$ 
	%
	It remains to verify that \(\delta_a(f) = 0\). To establish this, observe that for any integer $n \Ge 1$ and  $s \in \mathbb H_0$, we have  
	\beqn
	\sum\limits_{\substack{j = 1 \\ \tiny{\textbf{gpf}}(j) \Le p_n }}^\infty c_j j^{-s} = \zeta_n(s - r + 1) \left(1+ \frac{2^{r-1-s}}{1 - 2^{-s}} + \frac{3^{r-1-s}}{1 - 3^{-s}} + \cdots + \frac{p_n^{r-1-s}}{1 - p_n^{-s}}\right),
	\eeqn
	where \(\zeta_n\) is given by equation \eqref{zeta-n}.   
	From this, it follows immediately that $\delta_a(f) \Le 0$. Furthermore, since $c_{2^m} \Ge 2^{r-1}$ for all integers $m \Ge 1$, the series  
	\beqn
	\sum\limits_{\substack{j = 1 \\ \tiny{\textbf{gpf}}(j) \Le 2}}^\infty c_j = \sum_{m =0}^\infty c_{2^m}
	\eeqn
	diverges.  
	Thus, we conclude that \(\delta_a(f) = 0\), which implies that $\sigma_a(f) - \delta_a(f) = r$.  
	
	\noindent \textbf{Another approach to form a Dirichlet series for $r=0$}:
	To see a Dirichlet series satisfying $\sigma_a= \delta_a,$ let $g(s) = \sum\limits_{n=2}^\infty b_nn^{-s} \in \mathcal D$ be such that $b_n >0$ for all integers $n \Ge 2$ (this construction remains valid even if $\sigma_a(g)=-\infty$, meaning that $g$ is an entire Dirichlet series). Clearly, $g$ is a strictly decreasing function on $(\sigma_a(g), \infty)$ and $g(\sigma) \rar 0$ as $\sigma \rar \infty$. For $\alpha > 0,$ consider the function $f$ given by
	\beqn
	f(s): = \frac{1}{\alpha-g(s)}.
	\eeqn
	Assume that there exists $\rho \in (\sigma_a(g), \infty)$ such that $g(\rho) = \alpha.$ 
	Then, we claim that
	\beq \label{sigma=delta=rho}
	\text{$f \in \mathcal D$ and $\sigma_a(f) = \delta_a(f) = \rho.$}
	\eeq
	Note first that such a real number $\rho$ is unique as $g$ is injective, and $\alpha > 0.$ Since $g(\rho)= \alpha$ and $g$ is decreasing,  we have that $g(\rho + \epsilon) < \alpha$ for every $\epsilon > 0,$ and hence  
	\beqn
	|\alpha-g(s)| \Ge \alpha-g(\rho + \epsilon) > 0, \quad s \in \mathbb H_{\rho+ \epsilon}.
	\eeqn
	Also, since $g$ is absolutely convergent at $\rho,$ by \cite[Theorem~4.4.3]{QQ}, $f \in \mathcal D$ with $\sigma_a(f) \Le \rho,$ so that write 
	\beq
	\label{coeff-ii}
	f(s) = \sum_{n=1}^\infty c_n n^{-s}.
	\eeq 
	Furthermore, since $f$ blows up at $s = \rho,$ we get that $\sigma_a(f) = \rho.$ This also yields that $\delta_a(f) \Le \rho.$ 
	It now remains to verify that $\delta_a(f) \Ge \rho.$ 
	For every integer $n \Ge 1,$ let
	\beqn
	f_n(s): = \sum\limits_{\substack{j = 1 \\\tiny{\mbox{\bf gpf}}(j) \Le p_n }}^\infty c_j j^{-s} ~~\text{and}~~g_n(s): = \sum\limits_{\substack{j = 2 \\\tiny{\mbox{\bf gpf}}(j) \Le p_n }}^\infty b_j j^{-s}.
	\eeqn
	Then, for every $n \in \mathbb N,$ we have that $\sigma_a(f_n) \Le \sigma_a(f) = \rho$ and 
	\beq
	\label{fn-case}
	f_n(s) = \frac{1}{\alpha-g_n(s)}, \quad s \in \mathbb H_\rho.
	\eeq
	If $\sigma_a(f_m) = \rho$ for some $m \in \mathbb N,$ then it follows that $\sigma_a(f_n) \Ge \rho$ for all $n \Ge m,$ and hence by \eqref{eventually-equal-lemma}, $\delta_a(f) \Ge \rho,$ completing the claim. Otherwise, assume that $\sigma_a(f_n)<\rho$ for all $n \in \mathbb N.$ Let $x_0 < \sigma_a(g)$ be a real number very close to $\sigma_a(g).$ Since $g(x_0)$ is divergent, there exists a positive integer $n_0 \equiv n_0(x_0)$ such that 
	$g_n(x_0) > \alpha$ for all $n \Ge n_0.$ Note that $g \Ge g_n$ on $(\sigma_a(g), \infty)$, $g_n$ is decreasing and $g(\rho) = \alpha$. Thus, there exists a real number $\rho_n < \rho$ such that 
	$
	g_n(\rho_n) = \alpha,~ n \Ge n_0.
	$
	This implies that $g_n(\rho_n + \epsilon) < \alpha$ for every $\epsilon > 0$ as $g_n$ is a decreasing function, and hence  
	\beqn
	|\alpha-g_n(s)| \Ge \alpha-g_n(\rho_n +\epsilon) > 0, \quad s \in \mathbb H_{\rho_n+ \epsilon}.
	\eeqn
	Also, since $g_n$ is absolutely convergent at $\rho_n,$ we get that $\sigma_a(f_n) \Le \rho_n$ for all integers $n \Ge n_0.$ 
	Moreover, since $f_n$ blows up at $s = \rho_n$ (see \eqref{fn-case}), 
	\beqn
	\text{$\sigma_a(f_n) = \rho_n$ for all integers $n \Ge n_0.$}
	\eeqn
	Further,
	note that $\rho_n \Le \rho$ for all integers $n\Ge n_0$ and $\{\rho_n\}_{n=1}^\infty$ is an increasing sequence, which yields that $\rho_n \rar \rho_0 \Le \rho$ as $n \rar \infty.$ By using the fact that $g_n$ decreases, we get that
	\beq\label{decreasing-prop}
	g_n(\rho) \Le g_n(\rho_0) \Le g_n(\rho_n) =\alpha, \quad n \Ge n_0.
	\eeq
	Letting $n \rar \infty$ in \eqref{decreasing-prop} leads to $\lim_{n \rar \infty}g_n(\rho_0) = \alpha,$
	which yields that $g$ also converges at $\rho_0$ with $g(\rho_0) = \alpha.$ Thus, since $g$ is  injective over $(\sigma_a(g), \infty)$, it follows that $\rho_0= \rho,$ and hence $\rho_n$ converges to $\rho$ as $n \rar \infty.$ This combined with \eqref{eventually-equal-lemma} implies that $\delta_a(f) = \rho$ completing the proof of \eqref{sigma=delta=rho}. Therefore, we finally have that $\sigma_a(f) - \delta_a(f) = 0.$
	
	(ii) Let $r \Ge 1.$ For any prime $p$ and $m \in \mathbb N,$ define $c_{p^m} = p^{r-1}$ and extend it multiplicatively for all integers $n \Ge 1.$  
		Consider $f(s) = \sum_{n=1}^\infty c_n n^{-s}.$ Since $c_n > 0$ for all $n \in \mathbb N,$  
		\beqn
\sum_{n=1}^\infty  \frac{1}{p_n} = \sum_{n=1}^\infty c_{p_n} {p_n}^{-r}  \Le f(r),
		\eeqn
		which together with the fact that $\sum_{n=1}^\infty  \frac{1}{p_n}$ diverges (see, \cite[Eq.~(0.4)]{DGMS}) implies that $f$ also diverges at $s = r.$ Hence, 	
			\beq \label{one-estimate}
		\sigma_a(f) \Ge r.
		\eeq
		Moreover, since $r \Ge 1,$ we have that $c_{p^m} \Le p^{(r-1) m}$ for any prime $p$ and positive integer $m.$ This combined with multiplicativity of $\{c_n\}_{n \Ge 1}$ yields that $c_n \Le n^{r-1},~n \in \mathbb N,$ and hence for any $\sigma > r,$
		\beqn
		\sum_{n=1}^\infty c_n n^{-\sigma} \Le  \sum_{n=1}^\infty n^{-(\sigma+1 - r)} < \infty
		\eeqn
		implying that $\sigma_a(f) \Le r.$ Thus, by \eqref{one-estimate}, $\sigma_a(f) = r.$
		It now suffices to claim that $\delta_a(f) = 0.$ To see this, note first that for any integer $n \Ge 1,$
		\beqn
		 \prod_{j=1}^n \Big(1+\frac{p_j^{r-1-s}}{1-p_j^{-s}}\Big)
		\eeqn
		is a finite product of Dirichlet series absolutely convergent on $\mathbb H_0.$
		This combined with the multiplicativity of $\{c_n\}_{n \Ge 1}$ and the fundamental theorem of arithmetic yields that
		\beqn
		\sum\limits_{\substack{j = 1 \\\tiny{\mbox{\bf gpf}}(j) \Le p_n }}^\infty c_j j^{-s} =  \prod_{j=1}^n \Big(1+\frac{c_{p_j} p_j^{-s}}{1-p_j^{-s}}\Big) 
		= \prod_{j=1}^n \Big(1+\frac{p_j^{r-1-s}}{1-p_j^{-s}}\Big)
		\eeqn
		converges on $\mathbb H_0$ implying that
 $\delta_a(f) \Le 0.$ Also, since $\sum_{n=0}^\infty c_{2^n}$ is a divergent series, we obtain that $\delta_a(f) = 0.$ Finally, $\sigma_a(f) - \delta_a(f) = r.$
\end{proof}
    
    \begin{remark}
    	A careful analysis of the proof of Theorem~\ref{arbitrary-large}(ii) reveals that even if $r \Ge 0$, applying the same approach as in Theorem~\ref{arbitrary-large}(ii) yields $\delta_a(f) = 0$ and $\sigma_a(f) \Ge r$. However, determining the exact value of $\sigma_a(f)$ relied crucially on the assumption that $r \Ge 1$. Consequently, for the remaining case where $r \in [0,1)$, we do not know the precise value of $\sigma_a(f)$.
    	
    \end{remark}
    
    We now present a family of examples illustrating the second approach employed in the proof of Theorem~\ref{arbitrary-large}(i). A particular case of this has previously been presented in \cite{Mc-1}.	
    \begin{example}\label{example-inte}
    	For any integer $m \Ge 1,$ let $\rho_m \in (1, \infty)$ be the unique real number satisfying $\zeta^m(\rho_m) = 2.$ By \cite[Equation~1.2.2]{Ti},
    	\beqn
    	f_m(s) = \frac{1}{2- \zeta^m(s)} = \sum_{n=1}^\infty d_m(n) n^{-s}, \quad s, u \in \mathbb H_{\rho_m},
    	\eeqn
    	where, for every integer $n \Ge 2,$ $d_m(n)$ denotes the number of ways of expressing $n$ as a product of $m$ factors $($the order also matters$).$ Therefore, by  Theorem~\ref{arbitrary-large}(i), $\sigma_a(f_m) = \delta_a(f_m) = \rho_m.$
    \end{example}
\section{An application of Theorem~\ref{arbitrary-large}}

 For $\{a_n\}_{n=1}^\infty \subseteq (0, \infty),$ let $f(s) = \sum_{n=1}^\infty a_n n^{-s}$ with $\sigma_a(f) < \infty.$ With this sequence, we now associate a positive semi-definite Dirichlet series kernel $\kappa$ given by 
 	\beq
 	 \label{kappa-d}
 	\kappa(s, u) = \sum_{n=1}^\infty a_n n^{-s -\overline{u}},\quad s, u \in \mathbb H_{\frac{\sigma_a(f)}{2}}.
 	\eeq
 	Clearly, $\kappa$ converges absolutely on $\mathbb H_{\frac{\sigma_a(f)}{2}} \times \mathbb H_{\frac{\sigma_a(f)}{2}},$ and does not extend on any larger domain.
 	By the Moore's theorem \cite[Theorem~2.14]{P-R}, there exists an unique reproducing kernel Hilbert space $\mathscr H_\kappa$ associated with $\kappa,$ which is given by 
 	\beqn
 	\mathscr H_\kappa = \Big\{\sum_{n=1}^\infty b_n n^{-s} ~~:~~ \sum_{n=1}^\infty \frac{|b_n|^2}{a_n} < \infty\Big\}.
 	\eeqn
 	By Cauchy-Schwarz inequality, every function in $\mathscr H_\kappa$ converges absolutely on $\mathbb H_{\frac{\sigma_a(f)}{2}},$ and hence \beqn
 	\sup\{\sigma_a(h) : h \in \mathscr H_\kappa\} \Le \frac{\sigma_a(f)}{2}.
 	\eeqn
 	Furthermore, $\kappa(\cdot, \frac{\sigma_a(f)}{2}+\epsilon) \in \mathscr H_\kappa$ for any $\epsilon > 0,$ which admits a pole at $s = \frac{\sigma_a(f)}{2}-\epsilon.$ Thus, 
 	\beq
 	 \label{abs_H_d}
 	\sup\{\sigma_a(h) : h \in \mathscr H_\kappa\} = \frac{\sigma_a(f)}{2}
 	\eeq
 	implying that the common domain of analyticity of $\mathscr H_\kappa$ (the largest right half-plane where all series from $\mathscr H_\kappa$ converges) is $\mathbb H_{\frac{\sigma_a(f)}{2}}.$ 
 	The following is a well known fact related to the multipliers of $\mathscr H_\kappa$ (see, \cite[Theorem~3.1]{St}).
 	\begin{proposition}\label{stetler_result_repeat}
 		Let $\kappa$ be the kernel defined by \eqref{kappa-d}. Then, the multiplier algebra  
 		$\mathcal M(\mathscr H_\kappa)$ of $\mathscr H_\kappa$ is contractively included in $H^\infty(\mathbb H_{\frac{\delta_a(f)}{2}}) \cap \mathcal D.$
 	\end{proposition}

 As a consequence of Theorem~\ref{arbitrary-large} and Proposition~\ref{stetler_result_repeat}, we gain some insights into the structure of the multiplier algebra.

    \begin{corollary}
    For any non-negative real number $r,$ let $f$ be the Dirichlet series formed in the proof of Theorem~\ref{arbitrary-large}. Let $\kappa: \mathbb H_r \times \mathbb H_r \rar \mathbb C$ be the kernel given by
    \beqn
    \kappa(s, u) = f(s+\overline{u}) = \sum_{m=1}^\infty c_m m^{-s-\overline{u}}, \quad s, u \in \mathbb H_{\frac{r}{2}}.
    \eeqn 
    \begin{itemize}
    	\item [$(i)$] If $\{c_m\}_{m \Ge 1}$ is determined either by \eqref{coeff-i} $($with $r \in [0,1))$ or \eqref{coeff-ii}, then $\mathcal M(\mathscr H_\kappa) \subsetneq H^\infty(\mathbb H_{0}) \cap \mathcal D.$
    	\item [$(ii)$] If $\{c_m\}_{m \Ge 1}$ is the multiplicative sequence formed as in proof (ii) of Theorem~\ref{arbitrary-large}, then  
   $\mathcal M(\mathscr H_\kappa) = H^\infty(\mathbb H_0) \cap \mathcal D$ (isometrically).
    \end{itemize} 
    \end{corollary}
	\begin{proof}
	 By Theorem~\ref{arbitrary-large}, we have that 
	 \beq
	 \label{sigma-delta}
	\text{ $\sigma_a(f) = r$ and $\delta_a(f) = 0$.}
	\eeq 
	 
	 (i)  By Proposition~\ref{stetler_result_repeat}, we already know that $\mathcal M(\mathscr H_\kappa) \subseteq H^\infty(\mathbb H_{0}) \cap \mathcal D.$ To see the strict containment, since $r = \sigma_a(f) \in [0,1)$ and $\mathcal M(\mathscr H_\kappa) \subseteq \mathscr H_\kappa$, it follows from \eqref{abs_H_d} and \eqref{sigma-delta} that
	\beqn
	\sup\{\sigma_a(h) : h \in \mathcal M(\mathscr H_\kappa)\}
	&\Le& 
	\sup\{\sigma_a(h) : h \in \mathscr H_\kappa\}\\
	&=& \frac{r}{{2}}\\
	 &<& \frac{1}{2} = \sup\{\sigma_a(h) : h \in H^\infty(\mathbb H_{0}) \cap \mathcal D\},
	\eeqn
	where the last equality holds by an application of \cite[Theorem~4.1]{DGMS}.
	Therefore, $\mathcal M(\mathscr H_\kappa) \subsetneq H^\infty(\mathbb H_{0}) \cap \mathcal D.$
	
	(ii) By assumption, $c_{p^m} = p^{r-1}$ which extends multiplicatively for all integers $n \Ge 1.$ Then, \cite[Corollary~4.2]{St} (also, see \cite[Theorem~4.1]{Sa}) combined with the condition \eqref{sigma-delta} yields the desired conclusion.
	\end{proof}

\vskip.4cm 
\textit{Acknowledgment:}
The author gratefully acknowledges the support received from the Indian Institute of Technology Bombay through an Institute post-doctoral fellowship. The author is grateful to O. F. Brevig for some helpful discussions and to H. Ahmed for carefully reading the manuscript. Finally, we thank the anonymous referee for some suggestions that improved the original presentation.

\end{document}